\documentclass{amsart}

\usepackage{amscd,amssymb,dsfont} 
\usepackage{url}
\usepackage{rotating}
\usepackage[utf8]{inputenc}
\usepackage{amsmath}

\usepackage[colorlinks=true]{hyperref}
\usepackage{mathrsfs}
\usepackage{bbm}
\usepackage{framed}

\usepackage{xcolor}

\newcommand{\red}[1]{\textcolor{black}{#1}} 
\newcommand{\reD}[1]{\textcolor{black}{#1}} 
 \newcommand{\green}[1]{\textcolor{black}{#1}}
 \newcommand{\purple}[1]{\textcolor{black}{#1}}

\usepackage{stmaryrd}
\usepackage{url}
\usepackage{pinlabel}

\usepackage{subfigure}

\usepackage[export]{adjustbox}
\usepackage{tikz-cd}

\makeatletter
\def\UTFviii@defined#1{%
	\ifx#1\relax
	!!FIXME!!%
	\else
	\expandafter#1%
	\fi
}
\makeatother

\DeclareFontFamily{OT1}{pzc}{}
\DeclareFontShape{OT1}{pzc}{m}{it}{<-> s * [1.10] pzcmi7t}{}
\DeclareMathAlphabet{\mathpzc}{OT1}{pzc}{m}{it}

\usepackage{syntonly}

\newcommand{\deltxt}[1]{}

\newtheorem{theorem}{Theorem}[section]
\newtheorem{corollary}[theorem]{Corollary}
\newtheorem{lemma}[theorem]{Lemma}
\newtheorem{proposition}[theorem]{Proposition}

\theoremstyle{definition}
\newtheorem{definition}{Definition}[section]

\theoremstyle{remark}

\newcommand{\eval}[2][\right]{\relax
	\ifx#1\right\relax \left.\fi#2#1\rvert}

\newcommand\bR{{\mathbb{R}}}

\newcommand\bZ{{\mathbb Z}}

\newcommand\bV{{\mathbb V}}

\newcommand\Hom{{\rm Hom}}
\newcommand\dev{{\bf dev}}

\newcommand\clo{{\rm Cl}}

\newcommand\ra{\rightarrow}

\newcommand\emp{\emptyset}

\newcommand\Aff{{\mathbf{Aff}}}

\newcommand\Aut{{\mathbf{Aut}}}
\newcommand\Idd{{\rm I}}

\newcommand\bv{{\mathbf{v}}}

\newcommand\SL{{\mathsf{SL}}}

\newcommand\SO{{\mathsf{SO}}}

\newcommand\GL{{\mathsf{GL}}}

\newcommand{\vv}{\mathbf{v}}
\newcommand{\vw}{\mathbf{w}}

\newcommand{\Uu}{\mathbf{U}} 
\newcommand{\Uc}{\mathbf{UC}}

\newcommand{\llrrV}[1]{
	\left|\mkern-2mu\left|#1\right|\mkern-2mu\right|}

\begin{document}
	
	\title[Partially hyperbolic flows]{
		Partially hyperbolic flows on flat vector bundles with 
an application to complete affine manifolds.} 

	\author[S. Choi]{Suhyoung Choi}
	\thanks{This research was supported by the Basic Science Research Program through the National Research Foundation
		of Korea(NRF) funded by the Ministry of Education(2019R1A2C108454412), 
		Author orcid number: 0000-0002-9201-8210}
	\address{Department of Mathematical Sciences \\ KAIST \\Daejeon 34141, South Korea}
	
	\email{schoi@math.kaist.ac.kr}
	
	%
	
	
	\date{\today}
	
	%
	%
	%
	%
	%
	%
	%
	%
	%
	
	
	
	
	
	%
	
	\begin{abstract}

Let $N$ be a manifold of dimension $m$ with a flat vector bundle given by 
a representation $\rho:\pi_1(N) \ra \GL(n, \bR)$
where $\pi_1(N)$ is finitely generated. 
The holonomy group 
$\rho$ is a \emph{$k$-partially hyperbolic holonomy representation} if the flat bundle pulled back over the unit tangent bundle of a sufficiently large compact submanifold of $N$ splits into expanding, neutral, and contracting subbundles along the geodesic flow, where the expanding and contracting subbundles are $k$-dimensional with $k < n/2$.
Suppose that each element of $\rho(\pi_1(N))$ has an eigenvalue of norm $1$,
or, alternatively,  \red{$\rho$ has some singular values of subexponential growth
in terms of word length. }
We show that $\rho$ is a $P$-Anosov representation for a parabolic subgroup 
$P$ of $\GL(n, \bR)$ 
if and only if $\rho$ is a partially hyperbolic representation. 
We will primarily employ techniques from representation theory.
\purple{As an application, we will show that the equivalence holds} when $N$ is a complete affine $n$-manifold
and $\rho$ is a linear part of the holonomy representation. 
This had never been done over the full general linear group. 
	\end{abstract}
	\maketitle
	\tableofcontents


\section{Introduction}


\subsection{Notation}

Let $N$ be an $m$-manifold with a finitely generated fundamental group $\pi_1(N)$. 
Each flat vector bundle over $N$ can be described as 
$\bR^n_\rho(N):=\tilde N \times \bR^n/\sim$ where $(x, \vec{v}) \sim (\gamma(x), \rho(\gamma)(\vec{v}))$ for 
$\gamma \in \pi_1(N)$ where $\rho:\pi_1(N) \ra \GL(n, \bR)$ is a representation. 
A differential-geometry description of a flat bundle is a vector bundle with a flat linear connection. 

Let $M$ be a compact submanifold (with or without boundary) in $N$.
We assume that the inclusion map $M \hookrightarrow N$ induces a surjection on fundamental groups. The inverse image $\hat{M}$ of $M$ in $\tilde N$ is connected, and the deck transformation group $\Gamma_M$ of $\hat{M}$ is isomorphic to 
\green{$\pi_1(N)$. }
In these cases, we refer to $M$, the map $M \hookrightarrow N$, $\hat{M}$, the projection $p_{\hat{M}}: \hat{M} \ra M$, and $\Gamma_M$ as a \emph{fundamental group surjective {\rm (}FS\/{\rm }) submanifold}, an \emph{FS map}, an \emph{FS cover}, an \emph{FS covering map}, and an \emph{FS deck transformation group}, respectively.


A \emph{compact core} of a manifold is a compact submanifold homotopy equivalent to the ambient manifold. 
When a manifold is homotopy equivalent to a finite complex $K$, for $n=3$, Scott and Tucker \cite{ST89} proved the existence of a compact core, while for $n - \dim K \geq 3$, Stallings \cite{Stallings65} established this fact. Unfortunately, compact cores do not always exist (see Venema \cite{Venema80}). 
Although compact cores are preferable, we need to resort to FS submanifolds. This is fine for our purposes since we are working in coarse geometry essentially.


We will assume that $\partial M$ is convex and smooth in $N$. Let $d_M$ denote the path metric on \purple{$M$} induced from a Riemannian metric on $N$, and let $d_{\hat{M}}$ denote the path metric on $\hat{M}$ induced from it.

  



\begin{itemize} 
\item Let $\Uu M$ denote the space of direction vectors on $M$, i.e., 
the space of elements of the form 
$(x, \vec{v})$ for $x\in M$ and a unit vector $\vec{v} \in T_x M \setminus \{O\}$ 
with the projection $\pi_{\Uu M}: \Uu M \ra M$. 
\item Let $\Uu \hat{M}$ denote the space of direction vectors on $\hat{M}$
covering $\Uu M$ with the deck transformation group $\Gamma_M$.  
Let $\pi_{\Uu \hat{M}} : \Uu \hat{M} \ra \hat{M}$ denote the projection.
 \item  Let $d_{\Uu M}$ denote the path metric on $\Uu M$ obtained from
  the natural Riemannian metric
on the unit tangent bundle $\Uu M$ of $M$ 
obtained from the Riemannian metric of $M$
as a sphere bundle.   We will use 
$d_{\Uu \hat{M}}$ to denote the induced path metric on $\Uu \hat{M}$. 
\end{itemize} 


  
  A {\em complete isometric geodesic} in $\hat{M}$ is a geodesic that is an isometry of 
  $\bR$ into $\hat{M}$ equipped with \green{$d_{\hat{M}}$}. 
 Some but not all Riemannian geodesics are isometries. 
Actually, in \green{geodesic metric spaces}, these are identical to the notion of geodesics.
Since we are working on Riemannian metric spaces, we use the adjectives.  
  A {\em complete isometric geodesic} in $M$ is a geodesic that lifts to 
  a complete isometric geodesic in $\hat{M}$. 
\green{An} {\em isometric ray} in $\hat M$ is an isometric geodesic defined on  
$[0, \infty)$. 

  
%
  

We consider the subspace of $\Uu M$ where complete isometric geodesics pass. We denote this subspace by $\Uc M$, and call it a \emph{complete isometric geodesic \green{unit tangent} bundle}. The inverse image in $\Uu \hat{M}$ is a closed set denoted by $\Uc \hat{M}$. Since a limit of a sequence of isometric geodesics is an isometric geodesic, $\Uc M$ is compact and $\Uc \hat{M}$ is closed and locally compact. However, $\pi_{\Uu M}(\Uc M)$ may be a proper subset of $M$ for the
projection $\pi_{\Uu M}: \Uu M \ra M$, and
$\pi_{\Uu \hat{M}}(\Uc \hat{M})$ may be a proper subset of $\hat{M}$.

Let $\rho: \pi_1(N) = \Gamma_M \ra \GL(n, \bR)$ be a homomorphism.

There is an action of $\Gamma_M$ 
on $\Uc \hat{M} \times \bR^n$ given by 
\[ \gamma\cdot(x, \vec{v})  = (\gamma(x), \rho(\gamma)(\vec{v})) 
\text{ for } x \in \Uc\hat{M}, \vec{v} \in \bR^n.\]
Let $\bR^n_{\rho}$ denote the bundle over $\Uc M$
 given by taking the quotient $\Uc \hat{M} \times \bR^n$ by $\Gamma_M$. 

Recall that there is a geodesic flow $\phi$ in $\Uu M$ that restricts to one in $\Uc M$, also denoted by $\phi$. This lifts to a flow, still denoted by $\phi$, on $\Uc \hat M$. Hence, there is a flow $\Phi$ on $\Uc \hat M \times \bR^n$, which descends to a flow, also denoted by $\Phi$, on $\bR^n_{\rho}$.


\subsection{Main result} 
In this paper, we establish connections between the partial hyperbolicity in the context of dynamical systems and the \purple{Anosov} properties. 

\reD{Unlike the abstract dominated splitting discussed in many previous articles such as \cite{BPS} and \cite{KP22}, here we require actual splitting of smooth bundles and smooth flows, along with the presence of expanding and contracting properties in addition to the domination property.}



We need the following objects
since we need to work on manifolds and not just on groups. 
\begin{definition}[Partial hyperbolicity]\label{defn:phyp} 
	Suppose that $M$ is a compact Riemannian manifold with convex boundary.
	Let $\hat{M}$ be a regular cover of $M$ with the deck transformation group $G_M$.
	A representation $\rho: G_M \ra \GL(n, \bR)$ is {\em partially hyperbolic in a bundle sense 
over $M$} if 
	the following hold for a Riemannian metric on $M$: 
	\begin{enumerate} 
		\item[(i)] 
		There exist nontrivial $C^0$-subbundles $\bV_+, \bV_0$, and $\bV_-$ in $\bR^n_{\rho}$ 
		with fibers \[\bV_+(x), \bV_0(x), \bV_-(x) \subset \purple{\bR^n_{\rho}(x)}\] for the fiber $\purple{\bR^n_{\rho}}(x)$ of $\bR^n_{\rho}$ over each point $x$ of $\Uc M$
		invariant under the flow $\Phi_t$ of $\bR^n_{\rho}$ over 
		the geodesic flow $\phi_t$ of $\Uc M$. 
		\item[(ii)] $\bV_+(x), \bV_0(x)$ and $\bV_-(x) $ for 
		each $x \in \Uc M$ are independent subspaces, and their sum equals $\bR^n_{\rho}(x)$.
		\item[(iii)] For any fiberwise metric on $\bR^n_{\rho}$ over $\Uc M$, 
		the lifted action of $\Phi_{t}$ on $\bV_{+}$ (resp. $\bV_{-}$) is dilating (resp. contracting). 
		That is, for coefficients $A> 0, a> 0$, $A' > 0, a'>0$\/: 
		\begin{enumerate} 
			\item $\llrrV{\Phi_{-t}(\vv)}_{\bR^n_\rho, \phi_{-t}(m)} \leq 
A \exp(-a t)\llrrV{\vv}_{\bR^n_\rho, m}$ for $\vv \in \bV_+(m)$ as $t \ra \infty$. 
			\item $\llrrV{\Phi_{t}(\vv)}_{\bR^n_\rho, {\phi_{t}(m)}} \leq 
A \exp(-a t)\llrrV{\vv}_{\bR^n_\rho, m}$ for \green{$\vv \in \bV_-(m)$} as $t \ra \infty$. 
			\item (Dominance properties) 
			\begin{multline} \label{eqn:dominance} 
				\frac{\llrrV{\Phi_{t}(\vw)}_{\bR^n_\rho,\phi_{t}(m)}}{\llrrV{\Phi_{t}(\vv)}_{\bR^n_\rho, \phi_{t}(m)}} 
				\leq A' \exp(-a't)\frac{\llrrV{\vw}_{\bR^n_\rho, m}}{\llrrV{\vv}_{\bR^n_\rho, m}} \\
				\text{ for } \vv \in \bV_+(m), \vw \in \bV_0(m)
				\text{ as } t \ra \infty, \\
				 \text{ or for }
				 \vv \in \bV_0(m), \vw \in \bV_-(m)
				 \text{ as } t \ra \infty.
			\end{multline} 
		\end{enumerate} 
	\end{enumerate}
	We additionally require that $\dim \bV_+ = \dim \bV_- \geq 1$.  
	Here $\dim \bV_+$ is the {\em partial hyperbolicity index} of $\rho$
	\purple{for the associated bundle decomposition $\bV_+, \bV_0, \bV_-$ 
	and $\dim \bV_+ = \dim \bV_- < n/2$, so that $\dim \bV_0 \geq 1$. }
 Furthermore, $\bV_0$ is said to be the {\em neutral subbundle} of $\bR^n_{\rho}$. 
	(See Definition 1.5 of \cite{CP} or \cite{BLM}.)
We use the words ``in a bundle sense'' only as a convenience of this paper.)
\end{definition} 




The definition depends on the Riemannian metric, but not on $\llrrV{\cdot}_{\bR^n_{\rho}}$. If \purple{$G_M$} is word-hyperbolic, then we can replace $\Uc \hat M$ with a Gromov flow space which depends only on \green{$G_M$} up to quasi-isometry. (See Theorem \ref{thm:identify} and Gromov \cite{Gromov87} or Mineyev \cite{Mineyev}.)
\purple{Also, we leave as a question if Definition \ref{defn:phyp} forces 
$G_M$ to be word-hyperbolic.}


We say that $N$ has a \emph{partially hyperbolic linear holonomy group in 
a bundle sense} if
\begin{itemize}
\item there exists an FS submanifold $M$ in $N$ with convex boundary for a Riemannian metric on $N$, and
\item the linear part $\rho: \pi_1(N) = \Gamma_{M} \ra \GL(n, \bR)$ is a partially hyperbolic representation in a bundle sense over $M$. 
\end{itemize}

We denote by $a_1(A), \dots, a_n(A)$ the \green{nonincreasing} set of singular values of 
an $n\times n$-matrix $A$.
We denote by $\lambda_1(A), \dots, \lambda_n(A)$ the nonincreasing set of norms of 
eigenvalues of an $n\times n$-matrix $A$. (Here, the multiplicity is allowed.) 

Let $\log \lambda_1, \dots, \log \lambda_n$ denote the logarithmic eigenvalue functions defined on
the maximal \green{abelian} Lie algebra of $\mathfrak{gl}(n, \bR)$. 
For roots 
\[\theta=\left\{\log \lambda_{i_{1}}-\log \lambda_{i_{1}+1}, \dots, \log \lambda_{i_{m}}- \log \lambda_{i_{m}+1}\right\}, 
1 \leq i_{1}<\cdots<i_{m} \leq n-1,\] 
of $\GL(n, \bR)$, 
the parabolic
subgroup $P_{\theta}$ (resp. $\left.P_{\theta}^{-}\right)$ is 
the subgroup of block upper 
(resp. lower) triangular matrices in $\GL(n,\bR)$ with square diagonal blocks 
of sizes $i_{1}, i_{2}-i_{1}, \dots, i_{m}-i_{m-1}, n-i_{m}$ respectively. 
\red{We denote by $\theta^*$ the image of $\theta$ under the map $i \mapsto n-i$ for $i=1, \dots, n-1$.}

\purple{For a finitely generated group $G$, we denote by $w(g)$ the word length of $g$ in terms of a fixed finite generating set $S$ in $G$.}

\begin{theorem}\label{thm:main2} 
Let $N$ be an $m$-manifold with a finitely generated \purple{word-hyperbolic} fundamental group. 
	Let $\rho: \pi_1(N)\ra \GL(n, \bR)$ be a representation.
Assume one of the following: 
\begin{itemize} 
\item[(i)] \purple{$|\lambda_{i_g}(\rho(g))| =1 $} for all \green{$g \in \pi_1(N)$} where 
the index $i_g \in \{1, \dots, n\}$ may depend on $g$. 
\item[(ii)] \purple{For any sequence $g_i \in \pi_1(N)$, if $w(g_i) \ra \infty$, then $\frac{|\log a_{j_{g_i}}(\rho(g_i))|}{w(g_i)} \ra 0$ for some choices of 
the indices $j_{g_i}\in \{1, \dots, n\}$ depending on $g_i$.} 
\end{itemize} 

Then $\rho$ is $P$-Anosov for 
	any parabolic subgroup $P$ of $\GL(n, \bR)$ with 
	$P = P_\theta$ 
\red{for $\theta = \theta^*$}
containing $\log \lambda_k-\log \lambda_{k+1}, k \leq n/2$ 
	if and only if $\rho$ is partially hyperbolic of the index $k$ in a bundle sense. 
	In these cases $k < n/2$. 
\end{theorem}


\purple{We say that $\rho$ has {\em subexponential growth for some singular values} if (ii) holds.}  \purple{By Lemma 2.11 of \cite{BS2021} applied to the standard representation of 
$\GL(n, \bR)$, (ii) implies (i); however, we 
wrote out (ii) for the purpose of the usefulness of the result.
Also, the convergence in (ii) is uniform with respect to the word length since 
we can find a seqeunce not converging to $0$ otherwise.}

\red{It is easy to see that the conclusion is true for representations to 
$\SL_\pm(n, \bR)$ and for $k$ which is the least element of $\theta$ without our 
conditions.} \purple{If $\pi_1(N)$ is not simple, we can multiply any representation  
by a representation to $\bR^+$ so that a partial hyperbolicity cannot hold.}
We can ask if an Anosov representation is always partially hyperbolic up to multiplying by a positive scalar representation to $\bR$. We cannot answer this question since
controlling the singular values of representation elements is very hard in general, and
the only related results that we know of are due to Abel-Margulis-Soifer \cite{AMS95}, which is not applicable here
as far as we know. \red{See Kassel-Potrie \cite{KP25} for some discussions. 
However, only cases that are left are for representations 
whose singular values behave exponentially.}

\reD{The significance of this result is that we work with the full general linear group 
$\GL(n, \bR)$ instead of orthogonal groups. Also, Anosov representations correspond to only dominated representations and not actual partially hyperbolic flows as here. Promoting the dominated representation to Anosov ones is trivial for some Lie groups but not for general linear groups. }

\begin{corollary}\label{cor:affine}
Suppose that $N$ is a complete affine $n$-manifold with a finitely generated 
\purple{word-hyperbolic} fundamental group. 
Let $\rho:\pi_1(N) \ra \GL(n, \bR)$ be a linear part of the affine holonomy group. 
Then the conclusions of Theorem \ref{thm:main2} hold. 
\end{corollary} 

In the field of affine manifolds, 
the hyperbolic flow was already discovered by Goldman and Labourie \cite{GL12} for $3$-dimensional Margulis space-times, and Ghosh \cite{Ghosh17}, Ghosh-Treib \cite{GhoshTreib21}, and Danciger-Zhang \cite{DZ2019} prove these properties for affine actions on higher-dimensional affine spaces. However, the linear parts are always assumed to be in the orthogonal groups or Hitchin, where the properties are more or less immediate.

We wish to approach the conjecture that a closed complete affine manifold cannot have a word-hyperbolic fundamental group. One approach considered is using Oseledets-type theorems (See Chapter 8 of \cite{KH95}). However, generalizing our results to measurable settings is currently hindered by numerous technical issues. Discussions on these matters took place with researchers such as Nicholas Tholozan at BIRS, Banff, during the winter of 2020.






\subsection{Some historical comments on complete affine manifolds} 
Originally, Milnor posed the question whether a free affine group of rank $\geq 2$ could act freely and properly discontinuously on affine space. Margulis provided an affirmative answer by constructing examples in the affine $3$-space, where the linear parts belong to $\SO(2, 1)$. Drumm and Goldman subsequently explained this result using fundamental domains in a series of papers (see \cite{DDGS} for a survey). Sullivan \cite{Sullivan76} observed that the linear part of each affine group element must have $1$ as an eigenvalue and proved the Chern conjecture for complete affine manifolds \cite{KS75}. Klingler \cite{Klingler2017} later extended this to closed special affine manifolds.

Further developments include Ghosh \cite{Ghosh17,Ghosh18,ghosh2018avatars,Ghosh23}, Ghosh-Treib \cite{GhoshTreib21}, and Smilga \cite{Smilga14,Smilga16,Smilga18,Smilga22,Smilga22ii}, who generalized these results and produced higher-dimensional examples of proper and free affine actions. Abels, Margulis, and Soifer also contributed to related topics such as the Auslander conjecture (see Abels \cite{Abels01} for a survey). Danciger and Zhang \cite{DZ2019} investigated affine deformations of Hitchin representations, showing that they do not exist. Danciger, Kassel, and Gu\'eritaud \cite{DGK2020} provided \purple{the} first examples of fundamental groups of closed manifolds acting properly and freely on affine spaces.

Although these works are related to the themes of this paper, their specific aims differ. This summary only scratches the surface of this rapidly evolving field.

\subsection{Outline}

In Section \ref{sec:preliminary-p1}, we will introduce  affine structures and 
some basic facts on $\delta$-hyperbolic metric spaces. 
Suppose that $\Gamma_M$ is word-hyperbolic.
Then $\hat M$ has a Gromov hyperbolic Riemannian metric
by the Svarc-Milnor theorem. 
(See \cite{Milnor} and \cite{Svarc}. 
\purple{\v{S}varc is a transliteration of Albert Schwarz's name.
He published this theorem in Russian prior to Milnor. Currently, he is 
a professor at the University of California, Davis.)}
We discuss the ideal boundary of $\hat M$, identifiable with the Gromov boundary, and the complete isometric geodesics in $\hat M$. 
We relate the space of complete isometric geodesics to
the Gromov flow space. 

In Section \ref{sec:phyp}, 
we provide a review of the definition of the $P$-Anosov property of 
Guichard-Wienhard \cite{GW12}. For convenience for reductive groups, we will use 
the approaches of G\'ueritaud-Guichard-Kassel-Wienhard \cite{GGKW17ii} and Bochi-Potrie-Sambarino \cite{BPS}
using singular values of the linear holonomy group. (The theory of 
Kapovich, Leeb, and Porti \cite{KL18}, \cite{KLP18iii} is equivalent to these, but it is not so
tailored for our purposes here.)
We give various associated definitions. 

Section \ref{sec:PAnosov} has the main arguments. 
In Section \ref{sub:kconvexity}, we will relate the partial hyperbolic property in the singular-value sense to that in the bundle sense. 
In Section \ref{sub:converse}, we show that the linear part of 
the holonomy representation of a complete affine manifold 
is $P$-Anosov if and only if it is partially hyperbolic. 
First, we do this when the linear holonomy group
has a reductive Zariski closure. 
\purple{We use the fact that $1$ has to be an eigenvalue of each element 
or the subexponential growth of some singular values of
the linear holonomy groups}. We relate it
to singular values using the clever ideas of Danciger and Stecker,
the spectrum theory of Breuillard-Sert \cite{BS2021}, Benoist \cite{Benoist97}, \cite{Benoist98}, and the work of Potrie-Kassel \cite{KP22}.
We could generalize to the cases of nonreductive linear holonomy groups using 
the stability of the Anosov and partial hyperbolic properties. 
This will prove Theorem  \ref{thm:main2}. 

\purple{In Section \ref{sec:affine}, we apply these results to complete affine manifolds.  
proving Corollary \ref{cor:affine}.}

\subsection{Acknowledgments}
We thank Michael Kapovich for various help with geometric group theory and coarse geometry. This article began with some discussions with Michael Kapovich during the conference honoring the 60th birthday of William Goldman at the University of Maryland, College Park, in 2016. 
We thank Jeffrey Danciger and Florian Stecker for help with the $P$-Anosov properties, which yielded the main idea for the proof of Theorem \ref{thm:main2}. 
%
%
We also thank Herbert Abels, Richard Canary, Virginie Charette, Todd Drumm,
William Goldman, Fran\c{c}ois Gu\'eritaud, Fanny Kassel, 
Andr\'es Sambarino, Nicholas Tholozan, and Konstantinos Tsouvalas for various discussions helpful to this paper. 
We also thank BIRS, Banff, Canada, and KIAS, Seoul, where some of this work was done.
\reD{This paper is a joint work with Kapovich in the beginning and Danciger and Stecker later. Unfortunately, they did not wish to join as coauthors for various reasons.} 
\purple{Also, we used 
OpenAI to assist us with writing and Google Gemini to find related results and papers.}

\section{ Preliminary} \label{sec:preliminary-p1} 


\reD{We recall some standard well-known facts in this section to set the notations and so on. 
Also, we need a slight modification of the theories for our purposes. Most of the background on geometric groups and coarse geometry is from the comprehensive textbook by 
Dru\c{t}u and Kapovich \cite{DK2018}.}

The other purpose is to establish the Anosov properties for manifolds using coarse isometries since most of the literature works \green{with groups}.

\subsection{Convergences of geodesics} \label{sub:geoconv}  
Let $M$ be a manifold of dimension $n$ with a regular cover $\hat{M}$ and the deck transformation group $\Gamma_M$. We assume that $M$ has a Riemannian metric with convex boundary.

We say that a sequence of isometric rays (resp. complete isometric geodesics) $l_i$ \emph{converges} to an isometric ray (resp. complete geodesic) $l$ (denoted as $l_i \ra l$) if $l_i(t) \ra l(t)$ for each $t \in [0, \infty)$ (resp. $t \in \bR$).

By the Arzel\`a-Ascoli theorem and the properness of $\hat{M}$, any sequence of isometric rays (resp. complete isometric geodesics) converges to an isometric ray (resp. complete isometric geodesic) if the sequence of pairs of their $0$-points and directions at the $0$-points converges in $\Uu \hat{M}$. (See Section 11.11 of \cite{DK2018}.)

\subsection{Gromov hyperbolic spaces} \label{sub:Gromov} 


 Let us recall some definitions: 
A metric space is {\em geodesic} if every pair of points are connected by a geodesic, i.e., 
a path that is an isometry from an interval to the metric space.  
We will only work with geodesic metric spaces.
 Let $(X, d_X)$ and $(Y, d_Y)$ be metric spaces. A map $f: X \ra Y$ is called 
{\em  $(L, C')$-coarse Lipschitz} for $L, C > 0$ if 
 \[\green{d_Y(f(x), f(x'))} \leq L d_X(x, x') + C' \hbox{ for all } x, x' \in X.\] 
 A map $f: X \ra Y$ is an \green{{\em $(L, C)$-quasi-isometric embedding}} if 
 $f$ satisfies 
 \[L^{-1} d_X(x, x') - C \leq d_Y(f(x), f(x')) \leq L d_X(x, x')  + C \hbox{ for all }x, x'\in X.\] 
 A map $\bar f: Y\ra X$ is a {\em $C$-coarse inverse} of $f$ 
 for $C > 0$ if 
\[d_X(\bar f \circ f, \Idd_X) \leq C, \hbox{ and }
d_Y(f\circ \bar f, \Idd_Y)\leq C.\]  
A map $f:X \ra Y$ between metric spaces is called 
 a {\em quasi-isometry} if it is a coarse Lipschitz map \purple{that} admits a 
 coarse Lipschitz coarse inverse map.
(Note that these are not necessarily continuous, as we follow 
Dru\c{t}u-Kapovich \cite{DK2018}.) 

\begin{lemma} \label{lem:piUUM} 
Let $M$ be a compact Riemannian manifold with possibly nonempty boundary. 
	$\pi_{\Uu \hat M}$ is a quasi-isometry.
	\end{lemma} 
\begin{proof}
	Each fiber is uniformly bounded under $d_{\Uu \hat M}$.  
	\end{proof}

A metric space $(X, d)$ is {\em proper} if every metric ball has a compact closure. 
In other words, $d_p(\cdot) = d(p, \cdot)$ is a proper function $X \ra \bR$. Clearly,
a complete Riemannian manifold is a proper metric space.
We denote by $B^d(x, R)$ the ball of radius $< R$ in $X$ with the center $x \in X$. 

Let $X$ be a geodesic metric space with metric $d$. 
A {\em geodesic triangle $T$} is a concatenation of three geodesics $\tau_1, \tau_2, \tau_3$ where the indices are modulo $3$. 
The {\em thinness radius} of a geodesic triangle $T$ is the number 
\[ \delta(T) := \max_{j=1,2,3} \left( \sup_{p \in \tau_j} d(p, \tau_{j+1}\cup \tau_{j+2}) \right). \]

A geodesic metric space $X$ is called {\em $\delta$-hyperbolic} in the sense of Rips if 
every geodesic triangle $T$ in $X$ is $\delta$-thin. 

By Corollary 11.29 of \cite{DK2018} as proved by Section 6.3C of 
Gromov \cite{Gromov87}, 
the Gromov hyperbolicity is equivalent to the Rips hyperbolicity
for geodesic metric spaces.
We will use these concepts interchangeably.

We will assume that $\Gamma_M$ is word-hyperbolic for our discussions below. 
This is equivalent to assuming that
$\hat M$ is Gromov hyperbolic by the Svarc-Milnor lemma.

Two isometric rays $\rho_1$ and $\rho_2$ of $X$ are {\em equivalent} if 
$t\mapsto d(\rho_1(t), \rho_2(t))$ is bounded.  
This condition is equivalent to the one that their Hausdorff distance under $d$ is bounded. 
Assume that $X$ is $\delta$-hyperbolic. $X$ has a well-defined ideal boundary 
$\partial_\infty X$ as in Definition 3.78 of \cite{DK2018}; 
that is, the space of the equivalence classes of isometric rays. 
 (See \cite{CDP90} and \cite{GdH90}.) 
 We denote by $\partial^{(2)}_\infty X = \partial_\infty X \times \partial_\infty X \setminus \Delta$ 
 where $\Delta$ is the diagonal set. 

The constant $C$ below is called a {\em quasi-geodesic} constant. 
 \begin{lemma} \label{lem:twogeo} 
	Let $M$ be a compact manifold with 
	the induced path metric $d_{\hat M}$ on $\hat M$
	from a Riemannian metric of $M$.
	Let $\Gamma_M$ be the deck transformation group of $\hat M \ra M$.  
	Suppose that $\Gamma_M$ is word-hyperbolic. 

Then there exists a constant $C >0$ so that
for every pair of
	complete isometric geodesics $l_1$ and $l_2$ in $\hat M$
with the identical set of endpoints in $\partial_\infty \hat M$, the Hausdorff distance between 
$l_1$ and $l_2$ under $d$ is bounded above by $C$. 
\end{lemma} 
\begin{proof}
	Since $l_1$ and $l_2$ are both $(1, 0)$-quasi-geodesics, 
	this follows from Proposition 3.1 of Chapter 3 of \cite{CDP90}. 
\end{proof}

 	We put on $\hat M \cup \partial_\infty \hat M$ the 
 shadow topology, which is a first-countable Hausdorff topology
 by Lemma 11.76 of \cite{DK2018}. It is also compact according to Section 11.11 of \cite{DK2018}. 
 Since it is first countable, 
 we do not need to consider nets on the space but only sequences
 to understand the continuity of the real-valued functions. 

We denote by $\partial_G X$ the Gromov boundary of 
a $\delta$-hyperbolic geodesic metric space $X$. (See Section 11.12 of \cite{DK2018}.)
 By Theorem 11.104 of \cite{DK2018}, 
 there is a homeomorphism $h: \partial_\infty X
 \ra \partial_G X$ given 
 by sending the equivalence class $[\rho]$ of an isometric ray $\rho$
 to the equivalence class of $\{\rho(n)\}_{n \in \bZ_+}$. 
 We will identify these two spaces using this map. 

Let $\hat{\Gamma}_M$ denote the Cayley graph of $\Gamma_M$. 
$\hat{\Gamma}_M$ and $\hat M$ are quasi-isometric by 
the Svarc-Milnor lemma again. 
The Gromov boundary $\partial_\infty \hat M$ of $\hat M$ and 
can be identified 
with the boundary $\partial_\infty \Gamma$ of
$\hat{\Gamma}_M$ with the word metric
by Theorem 11.108 of \cite{DK2018}. 

For any Gromov hyperbolic geodesic metric space $X$, 
the set of $GX$ of complete isometric geodesics has a metric: 
for $g, h \in GX$, we 
define the metric $d_{GX}$ given by 
\begin{equation}\label{eqn:dGX} 
 d_{GX}(g, h) := \int^\infty_{-\infty} d_X(g(t), h(t)) 2^{-|t|} dt.    
\end{equation}
Notice that \purple{$d_X(g(t),h(t)) \leq 2|t| + C$ for a constant $C$}
since we can follow the geodesic $g$ backward and 
then forward in $h$ to obtain a distance $2|t|$. Thus, the above integral is always 
well defined. 
(See Gromov \cite{Gromov87}.)
%
%
Let $GX$ have this metric topology. 
(See Section 11.11 of \cite{DK2018}.)

\begin{lemma} \label{lem:identify} 
	Suppose that $\Gamma_M$ is Gromov hyperbolic. 
	Then there is a quasi-isometric homeomorphism $\mathcal{F}:\Uc \hat M \ra G\hat M$
	by taking a unit vector $\vec{u}$ at $x$ to a complete isometric geodesic 
	$\bR \ra \hat M$ passing $x$ tangent to $\vec{u}$. 
	\end{lemma} 
\begin{proof}
There is a map $G\hat{M} \ra \Uc \hat{M}$ given by sending \purple{each} complete isometric geodesic $g: \bR \ra \hat{M}$ to $(g(0), \vec{v}_0)$ where $\vec{v}_0$ is the unit tangent vector at $g(0)$. This map is clearly bijective. In Section III of \cite{Matheus90}, the map $G\hat{M} \ra \hat{M}$ defined by $g \mapsto g(0)$ is shown to be a quasi-isometry. Therefore, post-composing this map with a lift that sends each $g(0)$ to a vector in the fiber of $\Uc \hat{M}$ over $g(0)$ results in a quasi-isometry. For a sequence $\{g_i\}$ of geodesics in $\hat{M}$, if $d_{GX}(g, g_i) \to 0$, then $(g_i(0), g_i'(0)) \to (g(0), g'(0))$ since otherwise, we would obtain a positive lower bound for the integral \eqref{eqn:dGX}.

The inverse map $\Uc \hat{M} \ra G\hat{M}$ is also continuous due to the continuity of the exponential map and considerations involving \eqref{eqn:dGX}, where $d_X(g(t), h(t))$ grows sublinearly. The continuity of the integral values under $g$ and $h$ can be shown by cutting off for $|t| > N$ where $N$ is independent of $g$ or $h$ by \eqref{eqn:dGX} and an estimate.

	
	\end{proof}

%


\subsection{Flow space $G\hat M$ quasi-isometric to $\partial_\infty^{(2)}\hat M \times \bR$}

We assume that $\Gamma_M$ is word-hyperbolic. 
For a word-hyperbolic group $\Gamma$, we let $\hat \Gamma$ denote its Cayley graph.
For simplicity, we denote by $\partial_\infty \Gamma$ the boundary 
$\partial_\infty \hat \Gamma$. 
We denote \[\partial_\infty^{(2)}\Gamma
:=\partial_\infty\Gamma\times\partial_\infty \Gamma \setminus \{(t,t)\mid t\in\partial_\infty \Gamma\}.\]
We denote by $F\Gamma_M$ the Gromov 
geodesic flow space of $\hat{\Gamma}_M$ obtained by the following proposition. 
(Note that we identified $\partial_\infty \hat M$ with $\partial_\infty \Gamma$ in
Section \ref{sub:Gromov}.)

\begin{theorem}[Theorem 8.3.C of Gromov \cite{Gromov83}, Theorem 60 of Mineyev \cite{Mineyev}] 
	\label{thm:Mineyev} Let $\Gamma$ be a finitely generated word-hyperbolic group. 

Then there exists a proper hyperbolic metric space 
	$F\Gamma$ with the following properties\/{\em :}
	\begin{itemize} 
		\item[(i)] $\Gamma\times\bR\times\bZ/2\bZ$ acts on $F\Gamma$.
		\item[(ii)] The $\Gamma\times \bZ/2\bZ$-action is isometric.
		\item[(iii)] Every orbit $\Gamma\ra F\Gamma$ is a quasi-isometry. 
		\item[(iv)] The $\bR$-action is free, and every orbit $\bR \ra F\Gamma$ is a quasi-isometric embedding. 
		The induced map $F\Gamma/\bR \ra \partial_\infty^{(2)}\Gamma$ is a homeomorphism.
	\end{itemize}
\end{theorem}
We shall say that $F \Gamma$ is the {\em flow space} of $\Gamma$. 
In fact, $F \Gamma$ is unique up to a $\Gamma \times \bZ/2\bZ$-equivariant quasi-isometry sending $\bR$-orbits to $\bR$-orbits. 
We shall denote by $\phi_t$ the $\bR$-action on $F \Gamma$ and by 
$(\tau_+, \tau_-):F\Gamma \ra F\Gamma/\bR \cong \partial_\infty^{(2)}\Gamma$
the maps associating to an element of $F\Gamma$ 
the endpoint of its $\bR$-orbit.
Gromov identifies $F\Gamma$ with 
$\partial_\infty^{(2)} \Gamma \times \bR$ with a certain metric 
called the Gromov metric. 

\begin{proposition} \label{prop:Uc} 
	Let $M$ be a compact manifold with a covering map $\hat M \ra M$
	with a deck transformation group $\Gamma_M$. 
	Suppose that $\Gamma_M$ is word-hyperbolic.

	Then there is 
	a quasi-isometric surjective map 
	$\mathcal{E}: G\hat M \ra \partial^{(2)}_\infty \Gamma_M \times \bR$ with compact fibers. 
	The map obtained from composing with 
	a projection to the first factor is a map given by taking the endpoints of complete isometric geodesics. 
\end{proposition} 
\begin{proof} 
	Choose arbitrarily a base point $x_0$ in $\hat M$. 
	For each complete isometric geodesic $g: \bR \ra \hat M$, 
	let $g_{x_0}$ denote a projection point on $g(\bR)$ that is of 
	the minimal distance from $x_0$
	and let $\partial_+ g, \partial_-g \in \partial_\infty \Gamma_M$ denote
	the forward and backward endpoints,  
	and let $t(g) \in \bR$ denote $\pm d(g(0), g_{x_0})$ where we use $+$ 
	if $g(0)$ is ahead of $g_{x_0}$.  
	
	We define 
	$\mathcal{E}(g) = (\partial_+ g, \partial_- g, t(g))$. 
	The surjectivity follows from Proposition 2.1 of Chapter 2 of 
	\cite{CDP90} since $t$ is an isometry on each complete isometric  
	geodesic. 	The compactness of the fiber is implied by Lemma \ref{lem:twogeo}. 
	
	

	Champetier \cite{Champetier94} constructs the space 
	$G\hat{\Gamma}_M$ from $\hat{\Gamma}_M$ quasi-isometric to $\Gamma_M$ 
	following Gromov. 
Recall from Section \ref{sub:Gromov} that $\partial_\infty \hat{M} = \partial_\infty \Gamma_M$. Using the orbit map $\Gamma_M \ra \hat{M}$ sending $\gamma$ to $\gamma(x_0)$, we extend the map to $\hat{\Gamma}_M \ra \hat{M}$ by sending each edge to a shortest geodesic.

Let $I: G\hat{\Gamma}_M \ra G\hat{M}$ be a map that sends a geodesic $g$ in $\hat{\Gamma}_M$ to a complete isometric geodesic $g'$ in $\hat{M}$ with the same endpoints in $\partial_\infty \hat{M}$, and maps $g(0)$ to the nearest point on the image of $g'$.

The map $I': G\hat{M} \ra G\hat{\Gamma}_M$ can be defined by taking a complete isometric geodesic $g$ in $\hat{M}$ to a geodesic $g'$ in $\hat{\Gamma}_M$ in a similar manner, with $g(0)$ going to one of the elements of $\Gamma_M(x_0)$ nearest to it.

By Theorem 11.72 of \cite{DK2018}, which states that every isometric geodesic in $\hat{\Gamma}_M$ is a quasi-geodesic in $\hat{M}$ and vice versa, and considering that every complete isometric geodesic in $\hat{M}$ is uniformly bounded away from one in $\hat{\Gamma}_M$ in the Hausdorff distance, it follows that $G\hat{M}$ is quasi-isometric to $G\hat{\Gamma}_M$ via the maps $I$ and $I'$, using \eqref{eqn:dGX}.

	
	We define 
$\mathcal{E}': G\hat{\Gamma}_M \ra \partial_\infty^{(2)} \Gamma_M \times \bR$ 
	as we did for $\mathcal{E}$ above. 
	By Proposition 4.8 and (4.3) of \cite{Champetier94}, 
	$\mathcal{E}'$ sends 
	$G\hat{\Gamma}_M$ to 
	$\partial_\infty^{(2)} \Gamma_M \times \bR
	= \partial_\infty^{(2)} \hat{M}\times \bR$ as a quasi-isometric homeomorphism. 
	Since $\mathcal{E}(g, t)$ and $\mathcal{E}'\circ I'(g, t)$ are uniformly bounded away from 
	each other,  the result follows. 
	%
	%
	%
\end{proof}


The following shows that $\Uc\hat{M}$ can be 
used as our Gromov flow space. 

\begin{theorem}\label{thm:identify} 
	The map	
	$\mathcal{E}\circ \mathcal{F}:\Uc \hat M \ra \partial_\infty^{(2)} \Gamma_M \times \bR$ 
	is a quasi-isometry where each complete isometric 
	geodesic in $\Uc \hat M$ goes to 
	$(t_1,t_2)\times \bR$ for its pair $(t_1, t_2)\in \partial_\infty^{(2)} \Gamma_M$ 
    of endpoints. 
\end{theorem} 
\begin{proof} 
	Proposition \ref{prop:Uc} and Lemma \ref{lem:identify} imply
	the result. 
\end{proof}

\section{Partial hyperbolicity and $P$-Anosov properties} \label{sec:phyp} 


\reD{The purpose of this section is to set notations and recall standard facts since some of these are not too well-known to many geometric topologists not in this specific area.}

\subsection{Some background} \label{sub:background} 
We have to repeat some of the standard materials from
Section 2.2 of \cite{GGKW17ii} and Chapter 2 of \cite{GJT}. 
Consider a real reductive Lie group $G$ with a maximal compact subgroup $K$. 

We assume that $\mathrm{Ad}(G) \subset \Aut(\mathfrak{g})$ for simplicity. 
For example, this holds if $G$ is connected. We can usually reduce everything here to 
the case where $G$ is connected. 
Here, $G$ is an almost product of $Z(G)_0$ and $G_s$ 
where $Z(G)_0$ is the identity component of the center $Z(G)$ of $G$ 
and $G_s = D(G)$ the derived subgroup of $G$, which is semisimple. 

A parabolic subgroup $P$ of $G$ is a subgroup $P$ of the form 
$G \cap \mathbf{P}(\bR)$ for some algebraic subgroup $\mathbf{P}$ of an algebraic group $\mathbf{G}$ 
where $\mathbf{G}(\bR)/\mathbf{P}(\bR)$ is compact
and $\mathbf{G}(\bR) = G$. 

Let $\mathfrak{z}(\mathfrak{g})$ denote the Lie algebra of $Z(G)$, 
and $\mathfrak{g}_s$ the Lie algebra of $G_s$. 
Then $\mathfrak{g} = \mathfrak{z}(\mathfrak{g}) \oplus \mathfrak{g}_s$.
Letting $\mathfrak{t}$ denote the Lie algebra of $K$, 
we have $\mathfrak{g} = \mathfrak{t} \oplus \mathfrak{q}$ 
where $\mathfrak{q}$ is the $-1$-eigenspace of an involution $\theta$ preserving 
$\mathfrak{t}$ as the $1$-eigenspace.  
Let $\mathfrak{a} \subset \mathfrak{q}$ be the maximal abelian subalgebra, i.e., the Cartan subalgebra of 
$\mathfrak{g}$. Then $\mathfrak{a} = \mathfrak{z}(\mathfrak{g}) \cap \mathfrak{q} \oplus \mathfrak{a}_s$ 
where $\mathfrak{a}_s = \mathfrak{a} \cap \mathfrak{g}_s$, a maximal \green{abelian} subalgebra of 
$\mathfrak{q} \cap \mathfrak{g}_s$. 

We have a decomposition to $ad(\mathfrak{a})$-eigenspaces
 $\mathfrak{g}=\mathfrak{g}_{0} \oplus \bigoplus_{\alpha \in \Sigma} \mathfrak{g}_{\alpha}$.
 Here, 
 $\mathfrak{g}_{0}$ is the direct sum of $z(\mathfrak{g})$ and the centralizer of $\mathfrak{a}$ in $\mathfrak{g}_{s} $. 
 The set $\Sigma \subset \mathfrak{a}^{*}= \Hom_{\bR}(\mathfrak{a}, \bR)$ projects to a (possibly nonreduced) root system of $\mathfrak{a}_{s}^{*},$ and each $\alpha \in \Sigma$ is called a restricted root of $\mathfrak{a}$ in $\mathfrak{g}$.
 There is a subset 
$\Delta \subset \Sigma$ called a {\em simple system},  i.e. a subset such that any root is expressed 
uniquely as a linear combination of elements of $\Delta$ with coefficients all of the same sign; 
the elements of $\Delta$ are called simple roots.
There is a subset
$\Sigma^{+} \subset \Sigma$ which is the set of positive roots, i.e. 
 roots that are nonnegative linear combinations of elements of $\Delta$, where 
 $\Sigma=\Sigma^{+} \cup\left(-\Sigma^{+}\right)$ holds additionally. 

We define $P_\theta$ as the parabolic subgroup with the Lie algebra
\[\operatorname{Lie}\left(P_{\theta}\right)=\mathfrak{g}_{0} \oplus \bigoplus_{\alpha \in \Sigma^{+}} \mathfrak{g}_{\alpha} \oplus \bigoplus_{\alpha \in \Sigma^{+} \backslash \Sigma_{\theta}^{+}} \mathfrak{g}_{-\alpha}.\]

The connected components of $\mathfrak{a} \setminus \cup_{\alpha \in \Sigma} \ker(\alpha)$ 
are called {\em Weyl chambers} of $\mathfrak{a}$. 
The component where every $\alpha \in \Sigma_+$ is positive is denoted by $\mathfrak{a}^+$.

Let $G = \SL_\pm(n, \bR)$. 
The maximal abelian subalgebra 
$A_n$ of $\mathfrak{p}$ in this case
is the subspace of diagonal matrices with 
entries $a_1, \dots, a_n$ where $a_1 + \dots + a_n = 0$. 
Let $\log \lambda_i: A_n\ra \bR$ denote the projection to the $i$-th factor. 
$\Sigma$ consists of $\alpha_{ij} :=\log \lambda_i -\log \lambda_j$ for $i\ne j$, 
and \purple{the root space} $\mathfrak{sl}(n, \bR)_{\alpha_{ij}} = \{c E_{ij}| c\in \bR\}$ where $E_{ij}$ is a matrix with \purple{an} entry $1$ at $(i, j)$ and zero entries everywhere else. 

A positive Weyl chamber is given by
\[A_n^{+} =\{(a_1, \dots, a_n)| a_1\geq \dots \geq a_n, a_1+\cdots + a_n = 0\},\]
and we are given 
\begin{gather*} 
\Sigma^+ = \{\log \lambda_i -\log \lambda_j | i < j\} \hbox{ and }\\ 
\Delta = \{\alpha_i:= \log \lambda_i -\log \lambda_{i+1} | i=1, \dots, n-1\}.
\end{gather*}  

We recall Example 2.14 of \cite{GGKW17ii}: 
Let $G$ be $\GL(n, \bR)$ seen as a real Lie group. Its derived group is $G_{s}=D(G)=\SL(n, \bR) .$ 
We can take 
$\mathfrak{a} \subset \mathfrak{g l}(n, \bR)$ to 
be the set of real diagonal matrices of size $n \times n$. 
As above, for $1 \leq i \leq n,$ let $\log \lambda_{i} \in \mathfrak{a}^{*}$ 
be the evaluation of the $i$-th diagonal entry. 
Then 
\[\mathfrak{a} =\mathfrak{z}(\mathfrak{g}) \cap \mathfrak{a} \oplus \mathfrak{a}_{s} 
\hbox{ where }
\mathfrak{z}(\mathfrak{g}) \cap \mathfrak{a}=\bigcap_{1 \leq i, j \leq n}  \operatorname{Ker}\left(\log \lambda_{i}-\log \lambda_{j}\right)\] 
is the set of real scalar matrices, and $\mathfrak{a}_{s}=
\operatorname{Ker}\left(\log \lambda_{1}+\cdots+\log \lambda_{n}\right)$ 
is the set of traceless real diagonal matrices. The set of restricted roots of $\mathfrak{a}$ in $G$ is
$$
\Sigma=\left\{\log \lambda_{i}-\log \lambda_{j} \mid 1 \leq i \neq j \leq n\right\}.
$$
We can take $\Delta=\left\{\log \lambda_{i}-\log \lambda_{i+1} \mid 1 \leq i \leq n-1\right\}$ so that
$$
\Sigma^{+}=\left\{\log \lambda_{i}-\log \lambda_{j} \mid 1 \leq i<j \leq n\right\},
$$
and $\mathfrak{a}^{+}$ is the set of elements of $\mathfrak{a}$ whose entries are increasing. 
Let $\clo({\mathfrak{a}}^+)$ denote its closure. 
For 
\begin{multline} 
\theta=\left\{\log \lambda_{i_{1}}-\log \lambda_{i_{1}+1}, \dots, \log \lambda_{i_{m}}-\log \lambda_{i_{m}+1}\right\} \\ 
\hbox{ with } 1 \leq i_{1}<\cdots<i_{m} \leq n-1,
\end{multline} 
the parabolic
subgroup $P_{\theta}$ (resp. $\left.P_{\theta}^{-}\right)$ is the group of block upper (
resp. lower) triangular matrices in $\GL(n,\bR)$ with square diagonal blocks 
of sizes $i_{1}, i_{2}-i_{1}, \dots, i_{m}-i_{m-1}, n-i_{m}$. 
In particular, $P_{\Delta}$ is the group of upper triangular matrices in $\GL(n, \bR)$.

From now on, we will fix $\mathfrak{a}_n \subset \GL(n, \bR)$ 
as the set of real diagonal matrices of size $n \times n$. 
We will also regard $\theta$ as a subset of $\{1, \dots, n-1\}$ where 
each $i$ corresponds to $\log \lambda_i - \log \lambda_{i+1}$. 
Also, let $\mathfrak{a}_n^+ $ denote the open positive Weyl chamber.

\subsection{$P$-Anosov representations}  \label{sub:Panosov} 
Let $(P^{+}, P^{-})$ be a pair of opposite parabolic subgroups of $\GL(n, \bR)$, and 
let $\mathcal{F}^{{\pm}}$ denote the flag spaces $\GL(n, \bR)/P_{\pm}$.
Let $\chi$ denote the unique open $\GL(n, \bR)$-orbit of the product 
$\mathcal{F}^{+}\times \mathcal{F}^{-}$. 
The product subspace $\chi$ has two $\GL(n, \bR)$-invariant 
distributions $E^{\pm} := T_{x_{\pm}} \mathcal{F}^{\pm}$ for $(x^{+}, x^{-}) \in \chi$. 
\purple{Suppose that $\Gamma_M$ is a finitely generated word-hyperbolic group. }
Let $\phi_t$ be given as in Theorem \ref{thm:Mineyev} as the $\bR$-action on 
$F \Gamma_M$  whose orbits are quasi-geodesics. 
We denote by $(\tau_+, \tau_-): F\Gamma_M \ra F\Gamma_M/\bR \cong \partial_\infty \Gamma_M^{(2)}$ the maps associating to a point the \green{endpoints} of its $\bR$-orbit.


\begin{definition}[Definition 2.10 of \cite{GW12}]\label{defn:Anosov} 
\purple{Suppose that $\Gamma_M$ is a finitely generated word-hyperbolic group. }
	A representation $\rho: \Gamma_M \ra \GL(n, \bR)$ is {\em $(P^{+}, P^{-})$-Anosov} if 
	there exist continuous $\rho$-equivariant maps $\xi^+: \partial_\infty \Gamma_M \ra {\mathcal{F}}^+$, 
	$\xi^-:\partial_\infty \Gamma_M \ra {\mathcal{F}}^-$ such that:  
	\begin{enumerate} 
		\item[(i)] for all $(x, y)\in \partial_\infty^{(2)} \Gamma_M$, the pair 
		$(\xi^+(x), \xi^-(y))$ is transverse, and  
		\item[(ii)] for one (and hence any) continuous and equivariant family of norms 
		$(\llrrV{\cdot}_{m})_{m\in F\Gamma_M} $ on 
		\[(T_{\xi^+(\tau^+(m))} {\mathcal{F}}^+)_{m\in F\Gamma_M}
		\Big(\hbox{resp. } 
		(T_{\xi^-(\tau^-(m))} {\mathcal{F}}^-)_{m\in F\Gamma_M}\Big):\]
		\[ \llrrV{e}_{\phi_{-t}m} \leq A e^{-at} \llrrV{e}_{m}  \Big(\hbox{resp.} 
		\llrrV{e}_{\phi_{t}m} \leq A e^{-at} \llrrV{e}_{m} \Big).
		   \]
	\end{enumerate}
We call 	$\xi^\pm:\partial_\infty \green{\Gamma_M} \ra {\mathcal{F}}^\pm$ the {\em Anosov maps} associated with 
$\rho: \Gamma_M \ra G$. 
	(See also Section 2.5 of  \cite{GGKW17ii} for more details.)
\end{definition}

It is sufficient to consider only the case where $P^{+}$ is conjugate to $P^{-}$
by Lemma 3.18 of \cite{GW12}. (See also Definition 5.62 of \cite{KLP17}). 
In this case, $\rho$ is called {\em $P$-Anosov} for $P = P^+$. 
Note that this means that 
\begin{equation}  \label{eqn:theta*} 
\theta = \theta^*
\end{equation} 
where $\theta^*$ is the image $\theta$ under the map $i \mapsto n-i$, $i=1, \dots, n-1$.
Hence, this will always be true for the set $\theta$ occurring for $P_\theta$-Anosov 
representations below.
(See the paragraph before Fact 2.34 of \cite{GGKW17ii}.)

\subsection{$P$-Anosov representations and dominations}



Let $\rho: \Gamma_M \ra \GL(n, \bR)$ denote a representation. 
%
%
%
%
%
%
%
%
%
%
%
%
%
Let $w(g)$ denote the word length of $g$. 
Suppose that there exists an integer $k$,
$1 \leq k \leq n-k+1 \leq n$, 
so that the following hold for a constant $A, C > 0$: 
\begin{equation*}
 \frac{a_k(\rho(g))}{a_{k+1}(g)} \geq C\exp(A w(g)) \hbox{ and }
  \frac{a_{n-k}(\rho(g))}{a_{n-k+1}(g)} \geq C\exp(A w(g)).
\end{equation*} 
In this case, we say that $\rho$ is {\em $k$-dominated}
for $k \leq n/2$ 
(see Bochi-Potrie-Sambarino \cite{BPS}).





%
%


We will use $\llrrV{\cdot}$ to indicate the Euclidean norm in a maximal flat 
in a symmetric space $X$. (See Example 2.12 of \cite{KL18}.)
We will use the following notation: 
\begin{multline*} 
\begin{split} 
\vec{a}(g) & :=(a_1(g), \dots, a_n(g)), \\
\log \vec{a}(g) & := (\log a_1(g), \dots, \log a_n(g)) \in \clo(\mathfrak{a}^+_n)  \red{\hbox{(Cartan projection)}},\\ 
\log \vec{\lambda}(g)  &:= (\log \lambda_1(g), \dots, \log \lambda_n(g)) 
\in \clo(\mathfrak{a}^+_n) \red{\hbox{(Jordan projection)}}
\end{split} 
\end{multline*} 
for the singular values $a_i(g)$ and 
the modulus $\lambda_i(g)$ of the eigenvalues of $g$
for $g \in \GL(n, \bR)$.

%

Given a set $A$, 
we say that two functions $f, g:A \ra \bR$ are {\em compatible} if 
there exists a uniform constant $C>1$ such that 
$C^{-1} (x) < f(x) < C g(x)$ for all $x\in A$. 

\begin{lemma}\label{lem:Equiv} 
	Let a finitely generated \purple{word-hyperbolic} group $G$ have a representation $\rho:G \ra \GL(n, \bR)$.

	Then the following are equivalent\/{\em :} 
	\begin{itemize} 
	\item  $\rho$ is $P$-Anosov in the Gu\'eritaud-Guichard-Kassel-Wienhard sense \cite{GGKW17ii}
	for some parabolic group $P = P_\theta$ for an index set $\theta$ containing  $k$
	and $n-k$. 
	\item  $\rho$ is a $k$-dominated representation 
	for $1 \leq k \leq n/2$. 
	\end{itemize} 
\end{lemma} 
\begin{proof} 
	This follows \purple{from} Theorem 1.3(3) of \cite{GGKW17ii}
	since the set $\theta$ contains $k$ and $n-k$.  
\purple{Or one can use Propositions 4.5 and 4.9 of \cite{BPS}.}
\end{proof}

\section{The proof of Theorem  \ref{thm:main2}} \label{sec:PAnosov}

\subsection{$k$-convexity} \label{sub:kconvexity} 
\reD{We will relate the partial hyperbolic property in the singular-value sense to that in the bundle sense
	following \cite{BPS}.  We need this modification since we have to have actual expanding and shrinking under the flow.}

We continue to assume that $\Gamma_M$ is \purple{word-hyperbolic}. 
Let $\rho:\Gamma_M \ra \GL(n, \bR)$ be a representation. 
We say that the representation $\rho$ is {\em  $k$-convex} if there exist continuous maps 
$\zeta: \partial_\infty \Gamma_M \ra \mathcal{G}_k(\bR^n)$ and $\theta:\partial_\infty \Gamma_M\ra \mathcal{G}_{n-k}(\bR^n)$ such that the following hold:
\begin{description} 
	\item[(transversality)] for every $x,y\in\partial_\infty \Gamma_M,x\ne y$, we have $\zeta(x)\oplus\theta(y)=\bR^n$, and
	\item[(equivariance)] for every $\gamma \in \Gamma_M$, 
	we have \[\zeta(\gamma x) = \rho(\gamma)\zeta(x), \theta(\gamma x) = \rho(\gamma)\theta(x), x \in \partial_\infty \Gamma_M.\]
\end{description}


Using the representation $\rho$, it is possible to construct a linear 
flow $\psi_t$ over the geodesic flow $\phi_t$ of 
the Gromov flow space $\partial^{(2)}_\infty \hat M \times \bR$ as follows. 
Consider the lifted geodesic flow $\hat \phi_t$ on 
$\partial^{(2)}_\infty \hat M \times \bR$, and
define a linear flow on $\hat E := 
(\partial^{(2)}_\infty \hat M \times \bR)\times\bR^n$ by:
$\hat\psi_t(x,\vec{v})=(\phi_t(x), \vec{v})$ for $x\in\partial^{(2)}_\infty \hat M \times \bR$.
Now consider the action of $\Gamma_M$ on $\hat E$ given by:
\[ \gamma \cdot \left(x, \vec{v}\right):= 
\left(\gamma(x), \rho(\gamma)(\vec{v})\right). \] 
It follows that $\hat \psi_t$ induces a flow on $E_\rho:=\bR^n_{\rho}= \hat E/\Gamma_M$. 

When the representation $\rho$ is $k$-convex, by equivariance, there exists a $\phi_t$-invariant splitting of the form $E_\rho = Z \oplus \Theta$; 
it is obtained by taking the quotient of the bundles
\[\hat Z(x) = \zeta(x_+) \text{ and }
\hat \Theta(x) = \theta(x_-)
\text{ for } x \in \partial^{(2)}_\infty \hat M \times \bR\]
where $x_+$ is the forward endpoint of 
the complete isometric geodesic through $x$ and 
$x_-$ is the backward endpoint of the complete isometric geodesic through $x$. 

We say that a $k$-convex representation is {\em $k$-Anosov
in the bundle sense} if the splitting 
$E_\rho=Z\oplus\Theta$ is a dominated splitting for the linear bundle automorphism $\psi_t$, with $Z$ dominating $\Theta$ in the terminology of \cite{BPS}. 
This is equivalent to the fact that the bundle $\Hom(Z,\Theta)$ is uniformly contracted
by the flow induced by $\psi_t$.

Suppose that we further require for a $k$-dominated representation $\rho$
 \purple{where} $k \leq n/2$\/: 
\begin{itemize}  
	\item $a_p(\rho(g)) \geq C^{-1} \exp(A w(g))$ for $p \leq k$, and 
	\item $a_r(\rho(g)) \leq C \exp(-A w(g))$ for $r \geq n-k+1$
	\purple{and} some constants $C> 1, A> 0$. 
\end{itemize} 
Then we say that $\rho$ is {\em partially hyperbolic
	in the singular-value sense with the index $k$}. 
 
We obtain by Lemma \ref{lem:Equiv}:
\begin{lemma}\label{lem:phypAnos} 
	If $\rho$ is partially hyperbolic for an index $k$, $k \leq n/2$, in the singular-value sense, 
	then $\rho$ is $P_\theta$-Anosov for the index set $\theta$ containing $k$ and $n-k$. 
	\qed
\end{lemma} 

Recall $\hat M$ the cover of $M$ with the deck transformation group $\Gamma_M$. 
Recall from Section \ref{sub:Gromov}  that we can identify $\partial_\infty \Gamma_M$ for the Cayley graph $\hat \Gamma_M$ of $\Gamma_M$ with 
$\partial_\infty \hat M$ by Theorem 11.108 of \cite{DK2018}. 
Hence, we can identify $\partial_\infty^{(2)} \Gamma_M$ with $\partial^{(2)}_\infty \hat M$. 
The Gromov flow space $\partial_\infty^{(2)} \Gamma_M \times \bR$ is identified with 
$\partial^{(2)}_\infty \hat M \times \bR$. 
By directly following Proposition 4.9 of 
Bochi-Potrie-Sambarino \cite{BPS}, we obtain: 
\begin{theorem} \label{thm:bundle} 
Let $\rho: \Gamma_M \ra \GL(n, \bR)$ be a representation.
	Let $\rho$ be partially hyperbolic with an index $k$, $k < n/2$, 
	in the singular-value sense. 
	Then 
	$\rho$ is partially hyperbolic in the bundle sense with the index $k$.
\end{theorem} 
\begin{proof} 
By Lemmas \ref{lem:phypAnos}  and \ref{lem:Equiv}, 
	$\rho$ is $k$-dominated.

	Bochi, Potrie, and Sambarino \cite{BPS} 
	construct subbundles $Z$ and $\Theta$ of $E_\rho$ 
	where $Z$ dominates $\Theta$
	and $\dim Z = k, \dim \Theta = n-k$. 

	Now, $\rho$ is also \purple{$(n-k)$-dominated}. 
	By Lemma \ref{lem:phypAnos}, 
	we obtain a new splitting of the bundle $E_\rho = Z'\oplus \Theta'$ 
	over $\partial_\infty^{(2)}\hat M \times \bR$
	where $\dim Z' = n-k$ and $\dim \Theta'$ is $k$. 
	Furthermore, $Z'$ dominates $\Theta'$. 
	Then $Z$ is a subbundle of $Z'$, and 
	 $\Theta'$ is a subbundle of $\Theta$ by 
	Proposition 2.1 of \cite{BPS}. 
	%
	We now form bundles $Z, Z'\cap \Theta, \Theta'$ over $\partial_\infty^{(2)}\hat M \times \bR$. 
	The dominance property \eqref{eqn:dominance} of Definition \ref{defn:phyp}  follows from those of $Z, \Theta$ and $Z', \Theta'$. 
		
	There are exponential expansions 
	of $a_i(g)$ for $1\leq i \leq k$ by the partially 
	hyperbolic condition of the premise. 
	Now (iii)(a) of Definition \ref{defn:phyp} follows: 
	Define $U_p(A)$ as the sum of eigenspaces of $\sqrt{AA^*}$ corresponding to 
    \purple{the} $p$ largest eigenvalues. 
	Define $S_{n-p}(A):= U_{n-p}(A^{-1})$. 
	Since \reD{$\hat M/\Gamma_M$} is compact, Theorem 2.2 of \cite{BPS} shows 
	\[U_k(\psi^t_{\phi^{-1}_t(x)})\ra \hat Z(x) \hbox{ and } 
	S_{n-k}(\psi^t_x) \ra \hat \Phi(x) \hbox{ uniformly for $x \in \hat \Gamma_M$  as } t \ra \infty.\] 

	Since by the paragraph above Theorem 2.2 of \cite{BPS}, we obtain
	\[\psi^t(S_{n-k}(\psi^t_{\phi^{-1}_t(x)})^{\perp}) = U_k(\psi^t_{\phi^{-1}_t(x)}),\]
	where the unit vectors of $S_{n-k}(\psi^t_{\phi^{-1}_t(x)})^{\perp}$ goes to 
	vectors of norm at least the $k$-th singular value for $\psi_t$. 
		The angle between $\hat Z(x)$ and $\hat \Phi(x)$ is uniformly bounded below by 
a positive constant 
		by the compactness of \reD{$\hat M/\Gamma_M$}. 
		Hence, unit vectors in $\hat Z({\phi^{-1}_t(x)})$ have 
		components in $S_{n-k}(\psi^t_{\phi^{-1}_t(x)})^\perp$ whose lengths are uniformly bounded below by a positive constant 
		for sufficiently large $t$. 

	As $t \ra \infty$, $\phi_t$ passes the images of the fundamental domain under $\Gamma_M$. 
	Let $g_i$ denote the sequence of elements of $\Gamma_M$ arising in this way. 
		Since $a_k(g_i)$ grows exponentially, the expansion properties of $Z$ follow
\red{by the above paragraph}. 
	(iii) of  Definition \ref{defn:phyp} follows up to reversing the flow. 
	
			By the quasi-isometry of $\Uc \hat M$ with $\partial^{(2)}\hat M \times \bR$ in Theorem \ref{thm:identify},
	we construct our pulled-back partially hyperbolic bundles over $\Uc \hat M$.
%
%
%
	%
\end{proof}

\subsection{Promoting $P$-Anosov representations to  
partially hyperbolic ones}
\label{sub:converse} 

\reD{This subsection contains the main results of this paper. The complication is that the maximal abelian subgroup of the Zariski closure of the linear part of the holonomy group sometimes embeds not linearly but piecewise linearly. In addition, we have to understand the case where the Zariski closure is not reductive.}

\purple{By a {\em cone} of a subset $J$ of a vector space $V$, we mean 
the subset of nonnegative multiples of points of $J$. We denote 
it by $C(J)$. This is not necessarily convex.}. 

For a connected reductive Lie group $Z$,  
let $A_Z$ denote the maximal $\bR$-split torus of $Z$  with \purple{the} corresponding 
Lie algebra $\mathfrak{a}_Z$, and 
 let $\mathfrak{a}_Z^+$ denote the Weyl chamber of $\mathfrak{a}_Z$.
Also, we denote 
the Cartan projection by $\log \vec{a}_Z: Z \ra \clo(\mathfrak{a}^+_Z)$.
Let $G$ be a finitely generated Zariski dense group in $Z$. (See Section 1 of \cite{BS2021} 
and Section 2.4 of \cite{GGKW17ii}.)
We denote by $\log \lambda_{Z}: Z \ra  \clo(\mathfrak{a}^+_Z)$ the associated 
Jordan projection.
The {\em Benoist cone}  $\mathcal{BC}_Z(G)$ of 
$G \subset Z$ is the closure of the set of 
all positive linear combinations of 
$\log \vec{\lambda}_Z(g), g\in \Gamma$ in $\clo(\mathfrak{a}^+_Z)$ and $O$.
\purple{In particular,  $\mathcal{BC}_Z(G) \setminus \{O\}$ is path-connected and convex
in $\mathfrak{a}_Z$.
(See the paragraph before Theorem 1.2 of \cite{BS2021}.)}

We remark that we put $\mathfrak{a}_n$ as the diagonal matrices,   
and the Jordan projection here is same as the Cartan projection. 
We will use $\vec{\lambda}$ for $\mathfrak{a}_n$ following the convention.

We \green{worked out} the following proposition with the help of Danciger and Stecker: 
\begin{proposition}\label{prop:PAnosovStrict}  
Assume the following\/{\em :} 
\begin {itemize}
	\item  $G$ is a finitely generated group. 
	\item $\rho:G \ra \GL(n, \bR)$ is a
	$P_\theta$-Anosov representation 
	for some index set $\theta =\theta^*$ 
	containing $k$ and $n-k$ for 
	$1\leq k \leq n/2$.
	\item 
\begin{itemize} 
\item[(i)] $\rho(g)$ for each $g\in G$ has $1$ as the norm of an eigenvalue, 
or, alternatively, 
\item [(ii)]  \purple{For any sequence $g_i \in \pi_1(N)$, if $w(g_i) \ra \infty$, then $\frac{|\log a_{j_{g_i}}(\rho(g_i))|}{w(g_i)} \ra 0$ for some choices of 
the indices $j_{g_i}\in \{1, \dots, n\}$ depending on $g_i$.} 
\end{itemize} 
	\item The Zariski closure $Z$ of the image of $\rho$ is 
	a connected reductive Lie group. 
\end{itemize} 
	Then $\rho$ is partially hyperbolic with the index $k$ in the singular-value sense, 
	and $k < n/2$.   
	\end{proposition}
\begin{proof} 
	(I) We begin with some preliminary results.  
%
	 	The embedding $Z \hookrightarrow \GL(n, \bR)$
	 induces a map $\iota_Z:\clo(\mathfrak{a}_Z^+) \ra \clo(\mathfrak{a}^+_n)$ for the Weyl chamber
	 $\mathfrak{a}^+_n$ of $\GL(n, \bR)$. 
	 Let $\log\lambda_1, \dots, \log\lambda_n$ denote the 
	 coordinates of the Weyl chamber $\mathfrak{a}^+_n$, 
which are actually the standard coordinates.
	 Of course, $\iota_Z \circ \log \vec{a}_Z| 
	 \rho(G) = \log\vec{\lambda}|\rho(G)$. 


	 	Let $B:=\mathcal{BC}_Z(\rho(G))$ denote the 
	 Benoist cone of $\rho(G)$ in $Z$. 
Let $S$ be a generating subset of $\rho(G)$. Then 
\purple{$B$ is a cone of  the joint spectrum $J(S)$}
as stated immediately after Theorem 1.3 of \cite{BS2021}. 
\purple{That is, $B = C(J(S))$. }
 The image $\iota_Z(B)$ is not convex in general. 




Let $S \subset \rho(G)$ be a generating set. 
\purple{The condition (ii) implies (i) by Lemma 2.11 of \cite{BS2021}. 
Under the condition (i), each vector $\log \vec{\lambda}(\rho(g)), g\in G$ 
has a coordinate with the value $0$.
By Theorem 1.2 of \cite{BS2021}, we have 
\[ \clo\left(\bigcup_{m=1}^\infty \frac{1}{m}\log \vec{\lambda}_Z(S^m)\right) = J(S)\]}
\purple{Thus 
for each $x \in J(S)$, 
there is a sequence of elements of $x_j \in \frac{1}{m_j}\log\lambda_Z(S^{m_j}) $ converging to $x$. 
Since $B$ is the \purple{cone of $J(S)$} in $a^+_Z$, 
we find that there must be an index $i$ where 
$\iota_Z(x)_i = 0$ for every $x\in B$. }

From now on, the arguments are the same for both conditions.
We denote by 
\[\mathcal{I}(b)=\{ i \mid \log \lambda_i(\iota_Z(b)) = 0\} \text{ for } b \in B.\] 
which is a consecutive set under the first premise.
	 Let $S([1, ..., n])$ denote the collection of 
	 subsets of consecutive elements.
  
	 We define an integral-interval function,
	 \[\mathcal{I}: B \setminus \{O\} \ra S([1, \dots, n])\]
	 given by sending $b \in B$ to $\mathcal{I}(b)$. 
 Since the sequences of
ordered norms of eigenvalues of a convergent sequence of linear maps converge,
we have the following: 
	 \begin{equation} \label{eqn:limI} 
	 \lim_{i\ra \infty} \mathcal{I}(x_i) \subset \mathcal{I}(x) \text{ whenever } 
	 x_i \ra x, x_i, x\in B.
	 \end{equation}

Let $S$ denote a finite set of generators of $\rho(G)$ and their inverses. 
 By \purple{Theorem 1.3} of \cite{BS2021}, the sequence
 $\frac{1}{m} \log\vec{\lambda}_Z(S^m)\subset J(S)$ geometrically converges to 
 the compact convex set $J(S) \subset\clo(\mathfrak{a}_Z^+)$  as $m \ra \infty$ 
 (see Section 3.1 of \cite{BS2021}).
Moreover, the \purple{cone of $J(S)$} is the Benoist cone $B$. 
This implies that the set of directions of 
 $\log \vec{\lambda}_Z(\rho(G))$
 is dense in \purple{that of} $B$. (See also \cite{Benoist98}). 
 Let $S_n$ denote the unit sphere in the Lie algebra $\mathfrak{a}_n$ of $\GL(n, \bR)$
with respect to the standard coordinates.
 We obtain
 \[\log \lambda_k(g) -\log \lambda_{k+1}(g) \geq C \llrrV{\log \vec{\lambda}(g)}, g \in \rho(G),
 \hbox{ for a constant } C> 0 \]
 by Theorem 4.2(3) of \cite{GGKW17ii}. 
Hence, we obtain
 \begin{equation} \label{eqn:lklkpC} 
 	\log \lambda_k(b) -\log \lambda_{k+1}(b) \geq C 
 \end{equation} 
for every $b  \in B \setminus \{O\}$ going to an element of $\iota_Z(B) \cap S_n$ 
for a constant $C> 0$.
	 
	    (II) Our first major step is to prove
	 $\mathcal{I}(B \setminus \{O\}) \subset [k+1, n-k]$: 
	  Let $\iota_n$ denote the opposite involution of $\mathfrak{a}_n$. 
First, suppose that the rank of $Z$ is $\geq 2$, and hence
$\dim \mathfrak{a}_Z^+ \geq 2$. 

   \purple{Recall from the above of 
     this proposition that $B \setminus \{O\}$ is path-connected.}
	
A {\em diagonal subspace }
\[\Delta_{i, j} \subset \mathfrak{a}_n \text{ for } i\ne j, 
i, j=1, \dots, n,\] 
is a subspace of the maximal abelian algebra
$\mathfrak{a}_n$ given by 
$\log\lambda_i - \log\lambda_j=0$ for some indices $i$ and $j$. 
	The inverse images under $\iota_Z$ in $\clo(\mathfrak{a}_Z^+)$ of the 
	diagonal subspaces of $\mathfrak{a}_n$ may meet $B$. 

\purple{Hueristically speaking, 
under $\iota_Z$, $B$ is mapped piecewise linearly in $\mathfrak{a}_n$ and when it 
meets $\Delta_{i, j}$ for some $i, j$, it continues in the subspace obtained by 
applying a reflection about $\Delta_{i, j}$.}  

\purple{Let $\mathcal{I}'_B$ the collection of indices $(k, l)$, $k< l$, of 
	$\Delta_{k, l}$ so that $B$ meets $\Delta_{k, l}$ transversally}. 
	We call the closures of the components of 
	\[B - \bigcup_{(k, l)\in \mathcal{I}'_B} \iota_Z^{-1}(\Delta_{k, l}) \] 
	the {\em generic flat \purple{cones} of $B$}.  
	There are finitely many of these 
	to be denoted $B_1, \dots, B_m$ which are convex 
	\purple{cones} in $B$.  
\purple{These are mapped linearly inside the positive Weyl chamber of 
$\mathfrak{a}_n$.}

We make notes of two ``discrete continuity" properties due to Danciger and Stecker: 
\begin{itemize} 
\item	In the interior of each $B_i$, $\mathcal{I}$ is constant since the 
$\mathcal{I}$-value is never empty and the change  in $\mathcal{I}$-values 
will happen only in the inverse images of some of $\Delta_{k,l}$. 
	
	
\item	For adjacent regions, their images under $\mathcal{I}$ differ by
	adding or removing some top or bottom consecutive sets. 
\end{itemize} 
Heuristically speaking, worms leave ``traces".
	
	We now aim to prove that the following never occurs: 
	\begin{quotation} 
	There exists a pair of elements $b_1, b_2 \in B$
	with $b_1 \in B_p^o$ and $b_2 \in B_q^o$ for some $p, q$
	where $\mathcal{I}(b_1)$ has an element $\leq k$ and 
	$\mathcal{I}(b_2)$ has an element $> k$---(*).
	\end{quotation} 
(The idea is that $k, k+1$ both cannot be in a zero set since they differ by a constant due to Anosov properties.)

We use the discrete version of continuity: 
	Suppose that this (*) is true. 
	Since $B$ is convex with $\dim B \geq 2$, 
	there is a chain of adjacent polyhedral images 
	$B_{l_1}, \dots, B_{l_m}$ where $b_1 \in B_{l_1}^o$
	and $b_2 \in B_{l_m}^o$
	where $B_{l_j}\cap B_{l_{j+1}}$ contains a nonzero point.  
	For adjacent $B_{l_p}$ and $B_{l_{p+1}}$, 
	we have by \eqref{eqn:limI} 
\begin{equation} \label{eqn:intI} 
 \mathcal{I}(B_{l_p}^o)\cup \mathcal{I}(B_{l_{p+1}}^o) \subset \mathcal{I}(x)  \text{ for }
x\in \clo(B_{l_p})\cap \clo(B_{l_{p+1}}) \setminus \{O\}. 
\end{equation} 
	Considering $\mathcal{I}(B_{l_j}^o)$, $j=1, \dots, m$, and 
	the $\mathcal{I}$-values of the intersections of adjacent generic flat \purple{cones}, we have 
$\mathcal{I}(b') \ni k, k+1$ for some nonzero element $b'$ of $B$ by 
the connectedness of the $\mathcal{I}$-values as we change
generic flat \purple{cones} to adjacent ones according to 
\eqref{eqn:intI}.
Hence, we obtain
\begin{equation}\label{eqn:kk+1}  
\log \lambda_k(b') = 0 
\hbox{ and }\log \lambda_{k+1}(b')=0.
\end{equation}

	Since we can assume $\iota_Z(b') \in S_n$ for $b'$ in \eqref{eqn:kk+1}
	by normalizing, we obtain a contradiction to \eqref{eqn:lklkpC}.

	Hence, we proved (*), and  
only one of the following holds:
	\begin{itemize} 
	\item for all $b\in B$ in the interiors of generic flat \purple{cones},  every element of $\mathcal{I}(b)$ is $\leq k$ 
	or 
	\item for all $b \in B$ in the interiors of generic flat \purple{cones}, every element of 
	$\mathcal{I}(b)$ is $> k$.
	\end{itemize} 

Note that the opposition involution $\iota_n$ sends $\iota_Z(B)$ to itself
since the inverse map $g\ra g^{-1}$ preserves $Z$ 
and the set of eigenvalues of $g$ is sent to the set of 
the eigenvalues of $g^{-1}$.
	In the first case, by the opposition involution $\iota_n$ preserving $G$, 
	$\mathcal{I}(\iota_n(b))$ for $b \in B$ contains an element $\geq n-k+1$.
	This again contradicts the above. 
	Hence, we conclude $\mathcal{I}(b) \subset [k+1, \dots, n]$ for all $b \in B$
	in the interiors of generic flat \purple{cones}. 
Acting by the opposition involution and \green{using similar} arguments, we obtain 
$\mathcal{I}(b)\subset [k+1, \dots, n-k]$ for all $b \in B$ in the interiors of generic flat 
\purple{cones}. 

Furthermore, we can show that $\mathcal{I}(b) \subset [k+1, n-k]$ for all $b \in B \setminus \{O\}$ since otherwise, 
we have $\mathcal{I}(b') \ni k, k+1$ for some $b'\ne O$ by a similar reason to the above applied to lower-dimensional strata. 
This is a contradiction as before. 
In addition, the argument shows $[k+1, n-k] \ne \emp$ and $k < n/2$.

	 Now we consider the remaining case where $Z$ has rank $1$. 
Then the maximal abelian group $A_Z$ is $1$-dimensional. 
For each $g \in Z$, there is an 
index $i(g)$ for which $\log\lambda_{i(g)}(g) =0$. 
Thus, $\mathfrak{a}_Z$ is in the null space of $\log \lambda_i$ for some index $i$. 
Since $g \mapsto g^{-1}$ preserves $\mathfrak{a}_Z$, 
we have $\log \lambda_{n-i+1} = 0$ on $\mathfrak{a}_Z$. 
Since $\clo(\mathfrak{a}_Z^+) = \mathfrak{a}_Z \cap \clo(\mathfrak{a}^+)$, we also have 
$\log \lambda_{n-i+1}=0$ on $\mathfrak{a}_Z^+$. 
Hence, 
\[\mathcal{I}(B\setminus \{O\}) = [i, \dots, n-i+1] \hbox{ for some index } i, 0< i \leq n/2.\] 
Now by \eqref{eqn:lklkpC}, we must have $k < i$. This 
implies $\mathcal{I}(B \setminus \{O\}) \subset [k+1, n-k]$. 

	(III) To complete the proof, we show the 
	partial hyperbolicity property: 
We let $\log \vec{a}_Z|Z: Z \ra \mathfrak{a}_Z^+$ denote the Cartan projection of $Z$. 
Recall 
\[\iota_Z\circ \log \vec{a}_Z| \rho(G) 
= \log \vec{a}|\rho(G).\] 
Since $\rho(G)$ has a connected reductive Zariski closure, 
the sequence 
\[\left\{\frac{\log \vec{a}(\rho(g_l))}{w(g_l)}\right\}\]
\purple{has} limit points only in $\iota_Z(B)$
by Theorem 1.3 of \cite{BS2021}. 
By  the conclusion of Step (II), 
for any sequence $\{g_l\}$ in  $G$, 
\[
\frac{\min_{i=k+1,\dots,n-k} |\log a_i(\rho(g_l))|}{w(g_l)} \ra 0,
\]
uniformly in terms of $w(g_l)$. 
Hence, this implies that for any small constant $C> 0$, there is $N> 0$ such that 
\[a_i(\rho(g)) \geq \exp(-C w(g)), g \in G\] for at least one $i \in [k+1, \dots, n-k]$ 
provided $w(g) > N$.
This means that 
\[a_i(\rho(g)) \geq A \exp(-C w(g))\] for at least one $i \in [k+1, \dots, n-k]$ for all $g\in G$ 
for a constant $A > 0$ depending on $C$. 
By the $P_\theta$-Anosov property of $\rho$ and Lemma \ref{lem:Equiv}, we have 
\[\frac{a_k(\rho(g))}{a_{k+1}(\rho(g))}\geq A'\exp(C'w(g)), g \in \Gamma\] 
for some constants $A', C'> 0$. 
Since 
\[a_{k+1}(\rho(g)) \geq a_i(\rho(g)), i=k+1, \cdots, n-k, 
\hbox{ for every } g \in G,\] 
it follows that 
\[a_k(\rho(g)) \geq A'' \exp(C'' w(g))\] for some constants $C'', A''> 0$. 
Taking inverses, we also showed that $\rho'$ is partially hyperbolic 
in the singular-value sense with the index $k$. 

Finally, since $[k+1, n-k]$ is not empty, $k < n/2$.  
	\end{proof}

For a general representation $\phi: \Gamma \ra G$, we define the {\em semisimplification} 
$\phi^{ss}: \Gamma \ra G$ by post-composing $\phi$ with the projection to 
the Levi factor of the Zariski closure of $\phi(\Gamma)$.  
(See Section 2.5.4 of \cite{GGKW17ii} for details.)
Also, the Zariski closure of a finite-index group is reductive if so is the Zariski closure of the original discrete group.
(See Remark 2.38 of \cite{GGKW17ii}.)

	\begin{proof}[Proof of Theorem  \ref{thm:main2}] 
		\purple{Suppose that $\rho$ is $P_\theta$-Anosov in the sense of \cite{GGKW17ii}
where $k \in \theta$ and $\theta^\ast =\theta$. 
		Hence, $\Gamma:= \rho(\pi_1(N))$ is word-hyperbolic by 
the definition in \cite{GGKW17ii}.}  
	
		First, suppose that the Zariski closure of the image of the linear holonomy group is 
	reductive. 
	By Lemma \ref{lem:Equiv}, $\rho$ is a $k$-dominated
	representation of the index $k$. 
We take a finite-index normal subgroup $\Gamma'$ of $\Gamma$ 
\purple{so that} the Zariski closure $Z$ of $\rho(\Gamma')$ is a connected 
	reductive Lie group. 	 
	By Proposition \ref{prop:PAnosovStrict}, $\rho|\Gamma'$ is 
	partially hyperbolic in the singular-value sense with the index $k$, $k < n/2$.
   So is $\rho$ since $\Gamma$ is a finite-index extension of $\Gamma'$. 
	By Theorem \ref{thm:bundle}, $\rho$ is a partially hyperbolic 
	representation in the bundle sense with the index $k$.

	
	
	Now, we drop the reductive condition of the Zariski closure. 
Mainly, we have to prove the expanding and contracting properties in the subbundles. 
\purple{From now on, $P:= P_\theta$. The semisimplification $\rho^{ss}$ is $P$-Anosov by Proposition 1.8 of \cite{GGKW17ii},
and the Zariski closure of the image of $\rho$ is reductive. } 
\purple{We can conjugate $\rho$ arbitrarily close to $\rho^{ss}$ in 
$\Hom(\Gamma, \GL(n, \bR))$
since the orbit of $\rho^{ss}$ is the closed orbit in the closure of the orbit of $\rho$ under the conjugation 
action. (See Section 2.5.4 of \cite{GGKW17ii}. )}
\purple{We can assume that $\rho$ has the property (i) by Lemma 2.11 of \cite{BS2021}. 
Since eigenvalues do not change under conjugation, 
$\rho^{ss}(g)$	for each $g \in \Gamma$ has an eigenvalue equal to $1$. 
Hence (i) holds for $\rho^{ss}$. }
(See also the proof of Lemma 2.40 of \cite{GGKW17ii}.)
	By Proposition \ref{prop:PAnosovStrict} and Theorem \ref{thm:bundle}, 
	$\rho^{ss}$ is a partially hyperbolic representation in the bundle sense with the index $k$, $k < n/2$. 
	By Proposition 1.3 of \cite{GGKW17ii}, $\rho$ is $P$-Anosov with respect to 
	a continuous Anosov map $\zeta: \partial_\infty \Gamma \ra {\mathcal{F}}$ for the flag variety 
	$\mathcal{F} := \GL(n, \bR)/P$.

Let $\zeta_i: \partial_\infty \Gamma \ra {\mathcal{F}}$ denote the Anosov map for $\rho_i$ 
conjugate to $\rho$ converging to $\rho^{ss}$. 
Theorem 5.13 of \cite{GW12} generalizes to $\GL(n, \bR)$ 
since we can multiply a representation 
by a homomorphism $\Gamma \ra \bR^+$ not changing the $P$-Anosov property
to $\SL_\pm(n, \bR)$-representations. 
(See Section 2.5.3 of \cite{GGKW17ii}). 
Hence, Theorem 5.13 of \cite{GW12} implies that
$\zeta_i$ converges to $\zeta^{ss}: \partial_\infty \Gamma \ra {\mathcal{F}}$
the map for $\rho^{ss}$ in the $C^0$-topology. 

Let $\bV^{ss+}, \bV^{ss0}$, and $\bV^{ss-}$ denote the bundles of the partially hyperbolic decomposition of $\bR^n_{\rho^{ss}}$ 
as in Definition \ref{defn:phyp}. 
Let $\bV^{+}_i, \bV^{0}_i$, and $\bV^{-}_i$ denote the bundles of the decomposition of $\bR^n_{\rho_i}$ 
of respective dimensions $k$, $n-2k$, and $k$ obtainable from $\zeta_i$
since $\rho_i$ is $k$-dominated.  
The  single product bundle $\Uc \hat M \times \bR^n$ covers the direct sums of these bundles.
\begin{itemize} 
\item Let $\hat \bV^{ss+}, \hat \bV^{ss0}$, and $\hat \bV^{ss-}$ denote the cover of 
$\bV^{ss+}, \bV^{ss0}$, and $\bV^{ss-}$ \purple{respectively}. 
\item Let $\hat \bV^{+}_i, \hat \bV^{0}_i$, and $\hat \bV^{-}_i$ denote the cover of 
$\bV^{+}_i, \bV^{0}_i$, and $\bV^{-}_i$ \purple{respectively}.  
\end{itemize} 
Then the above convergence $\zeta_i \ra \zeta^{ss}$ means that 
\begin{equation} \label{eqn:convB}
\hat \bV^{+}_i(x) \ra \hat \bV^{ss+}(x), \hat \bV^{0}_i(x) \ra \hat \bV^{ss0}(x), 
\text{ and } \hat \bV^{-}_i(x) \ra \hat \bV^{ss-}(x) \text{ as } i \ra \infty
\end{equation}
in the respective Grassmannian spaces pointwise for each $x \in \Uc \hat{M}$. 
 
The convergence is also uniform over every compact subset of $\Uc\hat{M}$
since these vector spaces depend only on the endpoints of quasi-geodesics through $x$
and $\hat M$ is quasi-isometric to $\partial^{(2)}_\infty \Gamma_M \times \bR$. 

We may construct a sequence of the fiberwise metrics $\llrrV{\cdot}_{\bR^n_{\rho_i}}$  so that 
the sequence of the lifted norms of $\llrrV{\cdot}_{\bR^n_{\rho_i}}$  
uniformly converging to that of  $\llrrV{\cdot}_{\bR^n_{\rho^{ss}}}$ on $\Uc \hat M \times \bR^n$
over each compact subset of $\Uc \hat M$ 
using a fixed set of a partition of unity subordinate to a covering with trivializations.  

The flow $\phi^{ss}$ on $\bR_{\rho^{ss}}$ satisfies the expansion and contraction properties
in Definition \ref{defn:phyp}. 
Note that the lift $\Phi^{ss}$ of $\phi^{ss}$ and that $\Phi_i$ of the flow $\phi_i$ for $\bR^n_{\rho_i}$ in the product bundle 
$\Uc \hat M \times \bR^n$ are the same by construction
induced from trivial action on the second factor. 
Since $M$ is compact, there exists $t_0$ such that for $t> t_0$, 
\begin{itemize} 
\item $\Phi^{ss}_{t}| \bV^{ss+} $  is expanding by a factor $c^+> 1$ 
\item $\Phi^{ss}_{t}| \bV^{ss-}$ is contracting by a factor $c^- < 1$  with respect to 
the metric $\llrrV{\cdot}_{\bR^n_{\rho^{ss}}}$. 
\end{itemize} 
Hence, for sufficiently large $i$, 
the flow $\Phi_i$ on $\bR^n_{\rho_i}$  satisfies the corresponding properties
as well as the domination properties for $\llrrV{\cdot}_{\bR^n_{\rho_i}}$ 
by \eqref{eqn:convB}. 
Hence, $\Phi_i$ is partially hyperbolic with the index $k$.  

\purple{Conversely}, suppose that $\rho$ is partially hyperbolic with the index $k$. 
Proposition \purple{4.6} of \cite{BPS} shows that $\rho$ is $k$-dominated. 
Lemma \ref{lem:Equiv} \purple{(or Proposition 4.9 of \cite{BPS})} shows that $\rho$ is $P_\theta$-Anosov for $k \in \theta$. 
Also, $k < n/2$ since the neutral bundle exists. 
\end{proof}

\section{Complete affine manifolds} \label{sec:affine}

Let $G$ be a Lie group acting transitively and faithfully on 
a space $X$. 
A {\em $(G, X)$-structure} on a manifold $N$ is a maximal atlas of charts to $X$ so that 
the transition maps are in $G$. 
This is equivalent to
$N$ having a pair $(\dev, h)$ where 
$\hat N$  is a regular cover of $N$ with a deck transformation group $\Gamma_N$
such that
\begin{itemize} 
	\item $h:\Gamma_N \ra G$ is a homomorphism, called a {\em holonomy homomorphism}, and 
	\item $\dev: \hat N \ra X$ is an immersion,  
	called a {\em developing map}, satisfying 
	\[\dev \circ \gamma = h(\gamma)\circ \dev \]
 for each deck transformation  $\gamma \in \Gamma_N$.
\end{itemize}


We say that an $m$-manifold $N$ 
with a $(G, X)$-structure  is {\em complete} if $\dev: \hat N \ra X$ is a diffeomorphism. 
If $N$ is complete, 
we have a diffeomorphism 
$N \ra X/h(\Gamma_N)$, and $h(\Gamma_N)$ acts properly discontinuously and freely on $X$. Complete $(G, X)$-structures on $N$ are classified 
by conjugacy classes of representations $\Gamma_N \ra G$ with properly discontinuous and free actions on $X$. 
(See Section 3.4 of \cite{Thurston97}.)

Note that the completeness and compactness of $N$ have 
no relation for some $(G,X)$-structures 
unless $G$ is in the isometry group of $X$ with a Riemannian metric.
(The Hopf-Rinow lemma may fail.)

Now, we move on to the affine structures. 
Let $\mathds{A}^n$ be a complete affine space. 
	Let $\Aff(\mathds{A}^n)$ denote the group of affine transformations of $\mathds{A}^n$
	whose elements are of the form: 
	\[  x \mapsto A x + \bv   \]
	for a vector $\bv \in \bR^{n}$ and $A \in \GL(n, \bR)$. 
	Let $\mathcal{L}: \Aff(\mathds{A}^n) \ra \GL(n, \bR)$ denote the map that sends each element of $\Aff(\mathds{A}^n)$ to its linear part in $\GL(n, \bR)$.

		An affine $n$-manifold is an $n$-manifold with  an $(\Aff(\mathds{A}^n), \mathds{A}^n)$-structure. 
		A {\em complete affine $n$-manifold}
			is an $n$-manifold $N$ of the form $\mathds{A}^n/\Gamma$ for $\Gamma \subset \Aff(\mathds{A}^n)$
			acting properly discontinuously and freely on $\mathds{A}^n$.

	\begin{theorem}[Hirsch-Kostant-Sullivan \cite{KS75}]\label{thm:HKS}
	Let $N$ be a complete affine $n$-manifold.
	Let $\rho:\pi_1(N) \ra \GL(n, \bR)$ be the linear part of the affine holonomy representation $\rho'$.
	Then $g$ has an eigenvalue equal to $1$ for each $g$ in the Zariski closure $Z(\rho(\pi_1(N)))$ of $\rho(\pi_1(N))$. 
\end{theorem}
	\begin{proof} 
		We define the function $f: Z(\rho(\pi_1(N))) \ra \bR$ 
		by $f(x) = \det(x-\Idd)$ for each element $x$. 
		Suppose that $f(\rho(g)) \ne 0$, $g\in \Gamma_M \setminus \{\Idd\}$. Then $\rho(g)(x) + b = x$ has 
		a solution for each $b$ by Cramer's rule. 
		Hence, $\rho'(g)$ has a fixed point, which is a contradiction.
		We find that $f$ is zero in the Zariski closure, 
        which implies that each element has at least one eigenvalue equal to $1$.  
		\end{proof}

\begin{proof}[Proof of Corollary \ref{cor:affine}] 
By Theorem \ref{thm:HKS}, the main premise of Theorem \ref{thm:main2} is satisfied. 
Hence, the result follows. 
\end{proof}

\bibliographystyle{plain}
\bibliography{AA-PI}   

\end{document}